\documentclass[10pt, reqno]{amsart}
\usepackage{amscd, amssymb, amsmath,mathrsfs}
\usepackage{hyperref}
\usepackage{epsfig}
 \usepackage{verbatim}
\usepackage{dsfont}
\usepackage[usenames,dvipsnames]{xcolor}
\usepackage{xcolor}
\usepackage{marginnote}

\newtheorem{Df}{Definition}
\newtheorem{theorem}[Df]{Theorem}

\newtheorem{prop}[Df]{Proposition}
\newtheorem{lemma}[Df]{Lemma}

\newtheorem{remark}[Df]{Remark}

\newtheorem{corollary}[Df]{Corollary}
\numberwithin{Df}{section}
\numberwithin{equation}{section}

\newcommand{\tcoev}{\stackrel{\longrightarrow}{\operatorname{coev}}}
\newcommand{\tev}{\stackrel{\longrightarrow}{\operatorname{ev}}}
\newcommand{\ev}{\stackrel{\longleftarrow}{\operatorname{ev}}}
\newcommand{\coev}{\stackrel{\longleftarrow}{\operatorname{coev}}}

\newcommand{\Obj}{\operatorname{Obj}}

\newcommand{\brk}[1]{{\left\langle#1\right\rangle}}

\newcommand{\C}{\ensuremath{\mathbb{C}}}

\newcommand{\cat}{\mathscr{C}}
\newcommand{\ideal}{I}

\newcommand{\ob}{Ob(\mathcal{C})}

\newcommand{\End}{\operatorname{End}}
\newcommand{\Hom}{\operatorname{Hom}}
\newcommand{\tr}{\operatorname{tr}}
\newcommand{\Id}{\operatorname{Id}}

\newcommand{\kt}{$\Bbbk$\nobreakdash-\hspace{0pt}}
\newcommand{\kk}{\Bbbk}
\newcommand{\mt}{\operatorname{\mathsf{t}}}

\newcommand{\Proj}{\ensuremath{\mathcal{P}roj}}
\newcommand{\fg}{\ensuremath{\mathfrak{g}}}

\newcommand{\fh}{\ensuremath{\mathfrak{h}}}
\newcommand{\fb}{\ensuremath{\mathfrak{b}}}

\newcommand{\unit}{\ensuremath{\mathds{1}}}

\newcommand{\FK}{\kk}
\newcommand{\lega}{{\ensuremath{\color{RubineRed}\alpha}}}
\newcommand{\legb}{{\ensuremath{\color{RubineRed}\beta}}}
\newcommand{\PP}{{\ensuremath{{\color{RubineRed}\mathrm{P}}}}}
\renewcommand{\ss}{{\ensuremath{\color{RubineRed}\sigma}}}
\renewcommand{\tt}{{\ensuremath{\color{RubineRed}\tau}}}

\newcounter{exo} \newcounter{numexercice}
\renewcommand{\theexo}{\arabic{exo}}

\begin{document}
%N3:  \title{Twisted traces in pivotal categories}
\title{M-traces in (non-unimodular)  pivotal categories}

\author{Nathan Geer}
\address{Mathematics \& Statistics\\
  Utah State University \\
  Logan, Utah 84322, USA}
\thanks{Research of the first author has been partially supported by
  the NSF grants DMS-1452093 and DMS-1664387.   }
\email{nathan.geer@usu.edu}
\author{Jonathan Kujawa}
\address{Mathematics Department\\
University of Oklahoma\\
Norman, OK 73019}
\thanks{Research of the second author was partially supported by a Simons Foundation Collaboration Grant.}\
\email{kujawa@math.ou.edu}
\author{Bertrand Patureau-Mirand}
\address{LMBA, CNRS UMR 6205, Universit\'e de Bretagne-Sud, F-56000 Vannes, France }
\email{bertrand.patureau@univ-ubs.fr}
\thanks{Corresponding Author: Jonathan Kujawa.}\
\date{\today}

\begin{abstract} 
We generalize the notion of a modified trace (or m-trace) to the setting of non-unimodular categories.  
M-traces are known to play an important role in low-dimensional topology and representation theory, as well as in studying the category itself.   Under mild conditions we give existence and uniqueness results for m-traces in pivotal categories.  
 \end{abstract}
\maketitle
\setcounter{tocdepth}{1}
%\tableofcontents

\section{Introduction}

 \subsection{Background}

Tensor categories with compatible left and right duals have a notion of the categorical trace of a morphism.  These traces and the corresponding concept of dimension are a key tool in applications to low-dimensional topology, representation theory, and other fields.  However, it is often the case that categories of interest are not semisimple, the categorical traces vanish, and these constructions become trivial. In the past decade it became clear there exist nontrivial replacements for trace functions on non-semisimple ribbon and, more generally, pivotal categories (e.g., see \cite{GKP1,GKP2,GPV}).  We call these replacements \emph{modified traces} or \emph{m-traces}, for short. The study of m-traces leads to new, interesting quantum invariants of links and 3-manifolds as well as applications in  representation theory, Hopf algebras, Deligne categories, logarithmic conformal field theory, and other fields (e.g., see \cite{CMR,CG,DGP,GPT,CGP,CK,C,CH,BKN4,BCGP,BBG,BBGa,AS,Rup,Mur1,Lent,Ha}). 
However, until now the theory of m-traces was limited to unimodular categories (i.e., categories in which the projective cover and injective hull of the unit object coincide).  The primary goal of this paper is to generalize m-traces to the non-unimodular setting.  

 \subsection{Statement of main results}\label{SS:IntroStatMainResults} 
In what follows we highlight the main results.  For simplicity's 
sake, in the introduction we assume $\kk$ is an algebraically closed field and $\cat$ is a
pivotal, $\kk$-linear, locally-finite, tensor category.  Roughly speaking, this is a category with a tensor product, duals, and the morphism sets are finite-dimensional $\kk$-vector spaces.  See Section~\ref{S:Prelims} for precise definitions.   Such categories are ubiquitous.  For example, they appear:

\begin{itemize}
\item  as categories of finite-dimensional modules for finite-dimensional pivotal (quasi-)Hopf algebras;
\item  in the study of logarithmic conformal field theories (e.g., see \cite{Gab,TW});
\item  as fusion categories of categorical dimension zero (e.g., see \cite{EGNO}).
\end{itemize}
See Section~\ref{S:Examples} and \cite{GKP2} for specific examples of such categories.

 Given a fixed pair of objects $\lega$ and $\legb$ in $\cat$ 
 and a right ideal $\ideal$ (a certain kind of full subcategory),
  we define the notion of a right $\left(\lega , \legb \right)$-trace 
 indexed by the objects of $\ideal$. This is a particular kind of m-trace given by a family of $\kk$-linear functions,
$$\{\mt_V:\Hom_\cat(\lega\otimes V,\legb\otimes V)\rightarrow \FK \}_{V \in \ideal},$$
where $V$ runs over all objects of $\ideal$, 
and such that certain partial trace and cyclicity properties hold.  See Section~\ref{SS:traces} for a precise definition. In the case when $\lega$ and $\legb$ are both the unit object of $\cat$, then we recover the
unimodular m-traces of \cite{GKP1,GKP2,GPV}.

Our first main result is Theorem~\ref{T:ExistOfabtrace} in which we show, given an absolutely indecomposable object $\PP$ and objects $\lega$ and $\legb$ in $\cat$ with $\Hom_{\cat}(\lega , \PP)=\kk$ and $\Hom_{\cat}(\PP , \legb)=\kk$,  there exists a right $\left(\lega , \legb \right)$-trace on a certain 
 (possibly empty) right ideal 
$\ideal_{\lega}^{\legb}$. Furthermore, 
if either $\lega$ or $\legb$ is the unit object, then 
 Theorem~\ref{T:UniqunessofTrace} implies  this 
m-trace 
is unique up to scaling. 

Because of the generality of these existence and uniqueness results it is difficult to explicitly describe the ideal $\ideal_{\lega}^{\legb}$ or the functions $\mt_{V}$.
   However, there is a notable case where we can say more. 
%Added below -JK 
 If $\cat$ is an abelian category with enough projectives and
 $\PP$ is assumed to be the projective cover of the unit object, $\unit$, and if $\lega$ denotes a simple subobject of $\PP$, then our second main result shows the above theorem defines a unique, nontrivial, right $(\lega, \unit )$-trace on $\Proj (\cat)$, the full subcategory of projective objects of $\cat$.  
This is already interesting and powerful in the context of unimodular categories (i.e.\ when $\lega \cong \unit$).  It says any locally-finite, unimodular, pivotal category $\cat$  with enough projectives has an m-trace on $\Proj (\cat)$.  Previously this was only known in special cases, such as when $\cat$ is the category of representations for a finite group or when $\cat$ contains a simple projective object. For example, see \cite{BBGa,GKP2,GR}.

%Added references to the last sentence.  Dropped Integrable Systems as I gave up on finding an easy reference.  --JK
We end the paper with a discussion of how the notion of a right $(\lega , \legb )$-trace on a category leads to a natural generalization of the notion of a Calabi-Yau category.  Variations on Calabi-Yau categories play an important role in mathematical physics, algebraic geometry, the representation theory of finite-dimensional algebras, and the categorification of cluster algebras (e.g., see \cite{Costello,Keller,KR,KS} and references therein).

If $F,G$ are endofunctors of $\cat$, then we say $\cat$ is an $(F,G)$-twisted Calabi-Yau category if for all objects $U$ and $V$ there is a vector space isomorphism
\[
 \Hom_{\cat}(F(U),V) \cong \Hom_{\cat}(V,G(U))^{*}, 
\] which is functorial in both $U$ and $V$.     For example, an $(\Id_{\cat}, G)$-twisted Calabi-Yau structure on $\cat$ amounts to saying $G$ is a right Serre functor in the sense of Bondal-Kapranov \cite{BK}.  Just as a Calabi-Yau category is a categorical generalization of the notion of a symmetric Frobenius algebra, an $(F,G)$-twisted Calabi-Yau category generalizes the notion of an arbitrary Frobenius algebra. 

Our main theorem applied to $\Proj (\cat)$ can be reformulated as saying there is a twisted Calabi-Yau structure on $\Proj (\cat)$ for any pivotal, $\kk$-linear, locally-finite tensor category $\cat$. 
 As a consequence, $\Proj (\cat)$ admits a right Serre functor for any such category.  
See Theorem~\ref{T:twistedKYonProjforlocallyfinite}. As another consequence, every finite, $\kk$-linear, pivotal tensor category is equivalent to the category of finite-dimensional modules over a finite-dimensional Frobenius algebra.

\subsection{Applications}
One motivation for the development of m-traces is their use in constructing invariants in low-dimensional topology.  In particular, the second two  authors with Costantino and Turaev define generalized Kuperberg and Turaev-Viro invariants from certain unimodular pivotal tensor categories \cite{CGPT}.  One of the main ingredients in this construction is the existence of a non-degenerate m-trace.  This is one motivation for Theorem \ref{T:nondegeneratepairing} which says, under mild conditions, a right m-trace is always non-degenerate. This work is still within the context of unimodular categories.   An interesting future line of research is to use 
the m-traces of this paper
to construct generalized Kuperberg and Turaev-Viro invariants from non-unimodular pivotal tensor categories.

\subsection{Related work}
While working on this paper we learned 
A.\ Fontalvo Orozco and A.~M.\ Gainutdinov
were also defining a similar notion of m-traces with different techniques,  see \cite{FOG}. 
Their work generalizes the relation between the theory of integrals
for a Hopf algebra, $H$, with the m-trace on the
projective ideal $\Proj (H\text{-}\operatorname{mod})$ as established
in \cite{BBGa} in the unimodular case.

Shimizu recently introduced the notion of integrals for finite tensor categories \cite{Shim1,Shim2}.  It would be interesting  to generalize the results of \cite{BBGa} and Orozco-Gainutdinov to the categorical setting by relating the m-traces introduced here to the integrals of Shimizu.

\subsection{Acknowledgements}  The authors are pleased to thank the anonymous referee for  helpful comments.

\section{Preliminaries}\label{S:Prelims}

\subsection{Pivotal categories} 
We recall the definition of a pivotal tensor category, see for instance,
\cite{BW}. A \emph{tensor category} $\cat$ is a category equipped with a
covariant bifunctor $\otimes :\cat \times \cat\rightarrow \cat$ called the
tensor product, an associativity constraint, a unit object $\unit$, and left
and right unit constraints such that the Triangle and Pentagon Axioms hold.
When the associativity constraint and the left and right unit constraints are
all identities we say $\cat$ is a \emph{strict} tensor category. 
By Mac Lane's coherence theorem for pivotal tensor categories, every such category is equivalent (as a pivotal tensor category) to a strict one  (e.g., see \cite[Theorem 2.2]{NS}).  
%JK3: added below
However, when the results of this paper are applied in \cite{CGPT} the pivotal structure (in particular, the isomorphism in \eqref{E:DefphiV}) plays an important role.  For this reason we allow categories to have a nontrivial pivotal structure but are otherwise assumed to be strict.
The interested reader will easily extend all results to
arbitrary tensor categories or the reader may instead assume all categories are strict.  
In what follows we adopt the convention that $fg$ will denote the composition of morphisms $f \circ g$.

A strict tensor category $\cat$ has a \emph{left duality} if for each object
$V$ of $\cat$ there is an object $V^*$ of $\cat$ and morphisms
\begin{equation}\label{lele}
  \coev_{V} : \:\:   \unit \rightarrow V\otimes V^{*} \quad {\rm {and}} \quad
   \ev_{V}: \:\:
  V^*\otimes V\rightarrow \unit
\end{equation}
such that
\begin{align*}
  (\Id_V\otimes \ev_V)(\coev_V \otimes \Id_V)&=\Id_V & & {\rm {and}} &
  (\ev_V\otimes \Id_{V^*})(\Id_{V^*}\otimes \coev_V)&=\Id_{V^*}.
\end{align*}
A left duality determines for every morphism $f:V\to W$ in $\cat$ the dual (or 
transpose) morphism $f^*:W^*\rightarrow V^*$ by
$$
f^*=(\ev_W \otimes \Id_{V^*})(\Id_{W^*} \otimes f \otimes
\Id_{V^*})(\Id_{W^*}\otimes \coev_V),
$$
and determines for any objects $V,W$ of $\cat$, an isomorphism
$\gamma_{V,W}: W^*\otimes V^* \rightarrow (V\otimes W)^*$ by
$$
\gamma_{V,W} = (\ev_W\otimes \Id_{(V\otimes W)^*})(\Id_{W^*} \otimes \ev_V \otimes
\Id_W \otimes \Id_{(V\otimes W)^*})(\Id_{W^*}\otimes \Id_{V^*} \otimes
\coev_{V\otimes W}).
$$

Similarly, $\cat$ has a \emph{right duality} if for each object $V$ of $\cat$
there is an object $ V^\bullet$ of $\cat$ and morphisms
\begin{equation}\label{roro}
  \tcoev_{V} : \:\:   \unit\rightarrow V^\bullet\otimes V \quad {\rm {and}} \quad
  \tev_{V}:\:\:   V\otimes V^\bullet\rightarrow \unit
\end{equation}
such that
\begin{align*}
  (\Id_{V^\bullet}\otimes \tev_V)(\tcoev_V \otimes \Id_{V^\bullet})&=\Id_{V^\bullet}
  & & {\rm {and}} & (\tev_V\otimes \Id_{V})(\Id_{V}\otimes \tcoev_V)&=\Id_{V}.
\end{align*}
The right duality determines for every morphism $f:V\to W$ in $\cat$ the dual
morphism $f^\bullet:W^\bullet\rightarrow V^\bullet$ by
$$
f^\bullet=(\Id_{V^\bullet} \otimes \tev_W ) (\Id_{V^\bullet} \otimes f \otimes
\Id_{W^\bullet})( \tcoev_V \otimes \Id_{W^\bullet}),
$$
and determines for any objects $V,W$, an isomorphism $\gamma'_{V,W}:
W^\bullet\otimes V^\bullet \rightarrow (V\otimes W)^\bullet$ by
$$
\gamma'_{V,W}
= ( \Id_{(V\otimes W)^\bullet} \otimes \tev_V )(\Id_{(V\otimes W)^\bullet}
\otimes \Id_{V} \otimes \tev_W \otimes \Id_{V^\bullet} )(\tcoev_{V\otimes W}\otimes
\Id_{W^\bullet}\otimes \Id_{V^\bullet}).
$$

A \emph{pivotal category} is a tensor category with left duality
$\{{\coev_V}, \ev_V\}_{V \in \cat} $ and right duality
$\{{\tcoev_V}, \tev_V\}_{V \in \cat} $ which are compatible in the
sense that $V^*=V^\bullet$, $f^*=f^\bullet$, and
$\gamma_{V,W}=\gamma'_{V,W}$ for all $V, W, f$ as above.  Every
pivotal category has natural tensor isomorphisms
\begin{equation}\label{E:DefphiV}
\phi=\{\phi_V=(\tev_{V}\otimes\Id_{V^{**}})(\Id_V\otimes\coev_{V^{*}})\colon V\to
V^{**}\}_{V \in \cat}.
\end{equation}
%B2: added:
  As we said, we will assume without loss of generality that
  $\gamma_{V,W}=\Id_{V,W}$ but choose to keep $\phi$ non-trivial as this will be useful when the results of this paper are applied to the category of finite-dimensional representations
  of a non-involutive pivotal Hopf algebra in \cite{CGPT}.
% \begin{blue}
%   \%B: Can we assume that $\gamma$ is trivial as it is the case with
%   Hopf algebra rep. ?
%   \\
%   1) if we don't, there might be some missing $\gamma$ in the paper:
%   for example in equation 3.1.
%   \\
%   2) if we do we need to remove the $\gamma$ in the paper for ex in
%   proof of prop 3.4 and ad something like this:
%   \\ ************* \\
%   As we said, we will assume without loss of generality that
%   $\gamma_{V,W}=\Id_{V,W}$ but we will keep $\phi$ non trivial as it
%   is the case for the category of finite dimensional representations
%   of a non-involutive pivotal Hopf algebra.
% \end{blue}

% Added below -JK
If $\cat$ is a pivotal category, then the tensor functor is exact in
both variables \cite[Propostion 4.2.1]{EGNO}.

 We remind the reader of the existence of a diagrammatic calculus for pivotal tensor categories, see for example \cite[Chapter XIV]{Kas} or \cite{GPV}.  
For brevity's sake we choose to not use it here. 
%Nevertheless, many of the calculations in this paper are most easily understood when done diagrammatically. 
%Changed to below -JK
%The reader who prefers diagrammatic arguments may find it helpful to translate the arguments given here into the relevant pictures.

\subsection{Tensor categories}
Let $\kk$ be a commutative ring.
A \emph{tensor \kt category} is a tensor category $\cat$ which is enriched over the category of $\kk$-modules.  That is, 
%J2: Added
$\cat$ is additive, 
the hom-sets of $\cat$ are left $\kk$-modules, and the composition and tensor product of morphisms are $\kk$-bilinear.

An object $V$ of a tensor \kt category $\cat$ is \emph{absolutely irreducible} if $\End_\cat(V)$ is a free $\kk$-module of rank one%
% % B: added then removed
% genrated by $\Id_V$
% %
; that is, if the
% B : replace \kt homomorphism with ring to ensure the iso is Id ---> 1
% \kt homomorphism
ring homomorphism
$\kk \to
\End_\cat(X),\, k \mapsto k\, \Id_X$ is an isomorphism.  We identify
$\End_\cat(V)$ and $\kk$ via this map.  We always assume the unit
object, $\unit$, is absolutely irreducible.

We call an object $V$ of $\cat$ \emph{absolutely indecomposable} if 
\[
\End_{\cat}(V)/\operatorname{J}(\End_{\cat}(V)) \cong \FK.
\] 
Here $\operatorname{J}(\End_{\cat}(V))$ denotes the Jacobson radical of the endomorphism ring $\End_{\cat}(V)$. We say an absolutely indecomposable object is \emph{end-nilpotent} if the  Jacobson radical of its endomorphism algebra is nilpotent. 

%Added below -JK
Assume $\cat$ is abelian.  A simple object is one with no nontrivial subobjects.   An indecomposable object is one which cannot be written as a nontrivial direct sum of subobjects.  If $\kk$ is an algebraically closed field, then by Schur's lemma any simple object is absolutely irreducible and by Fitting's lemma any finite-length indecomposable object is absolutely indecomposable and end-nilpotent (see \cite[Lemma II.4.1]{Alperin}). 

  \subsection{Projective and injective objects}\label{SS:ProjInj}
%Changed this subsection --JK
Let $\cat$ be an abelian category.  Recall an object $P$ of $\cat$ is \emph{projective} if the functor $\Hom_\cat(P,-)\colon \cat \to \mathrm{Set}$ preserves epimorphisms, that is, if for any epimorphism $p\colon X \to Y$ and any morphism $f\colon P \to Y$ in $\cat$, there exists a morphism $g \colon P \to X$ in~$\cat$ such that $f=pg$. We denote by $\Proj(\cat)$ the class of projective objects of $\cat$.    An object of $\cat$ is \emph{injective} if it is projective in the opposite category $\cat^\mathrm{op}$. In other words, an object $Q$ of $\cat$ is injective if for any monomorphism $i\colon X \to Y$ and any morphism $f\colon X \to Q$ in $\cat$, there exists a morphism $g \colon Y \to Q$ in $\cat$ such that $f=gi$.

% \begin{blue}
%   \%B: in the following is $\cat$ a tensor \kt category ?
% \end{blue}
When $\cat$ is pivotal the dual of a projective object is again projective and, hence, projective and injective objects coincide (see \cite[Lemma~17]{GPV} or \cite[6.1.3]{EGNO}).  Thus in this case $\Proj(\cat)$ is also the class of injective objects of $\cat$. The projective cover of an object is unique up to non-unique isomorphism, if it exists.  We say $\cat$ has \emph{enough projectives} if every object in $\cat$ has a projective cover.

%The following paragraph needs to be worked on --JK
We call
% B2:
a \kt category
$\cat$ \emph{locally-finite} if, for every pair of objects $X,Y$ in $\cat$, $\Hom_{\cat}(X,Y)$ has a finite length composition series as a $\kk$-module and if every object has finite length.  If $\cat$ is also pivotal, then the projective cover of a simple object, $P$, is indecomposable and has a unique simple subobject which we call the \emph{socle} of $P$ (see \cite[Remark 6.1.5]{EGNO}).
%, and $\End_{\cat}(P)$ is   %N2:  with $\operatorname{J}(\End_{\cat}(P))$ consisting of nilpotent elements.  
%end-nilpotent.  
By definition, $\cat$ is \emph{unimodular} if the socle of the projective cover of $\unit$  is isomorphic to $\unit$.

\subsection{Invertible objects}\label{SS:invertibleobjects}
We call an object $X$ in $\cat$ \emph{invertible} if $\ev_{X}:X^{*}\otimes X \to \unit$ and $\coev_{X}:\unit \to X \otimes X^{*}$ are isomorphisms.  For example, $\unit$ is always an invertible object and in a finite tensor category the socle of the projective cover of $\unit$ is always an invertible object  (see \cite[Section 6.4]{EGNO}).

\section{Right \texorpdfstring{$(\lega,\legb)$}{(alpha, beta)}-Traces}

\subsection{Ideals}
By a \emph{right (resp.\ left) ideal} of $\cat$ we mean a full subcategory, $\ideal$, of $\cat$ such that:  
\begin{enumerate}
\item \textbf{Closed under tensor products:}  If $V$ is an object of $\ideal$ and $W$ is any object of $\cat$, then $V\otimes W$ (resp.\ $W \otimes V $) is an object of $\ideal$.
\item \textbf{Closed under retracts:} If $V$ is an object of $\ideal$, $W$ is any object of $\cat$, and there exists morphisms  $f:W\to V$,  $g:V\to W$ such that $g  f=\Id_W$, then $W$ is an object of $\ideal$.
\end{enumerate}

An \emph{ideal} of $\cat$ is a full subcategory of $\cat$ which is both a right and left ideal.  For example, the full subcategory whose objects are the class of projective objects, $\Proj(\cat)$, is an ideal by \cite[Lemma~17]{GPV}.

\subsection{Traces} \label{SS:traces}
 Let $\cat$ be a pivotal $\kk$-category.  
A \emph{right partial trace} (with respect to $W$) is the map
$\tr_r^W \colon \Hom_\cat(V \otimes W, X \otimes W) \to \Hom_\cat(V,X)$
defined, for $g \in  \Hom_\cat(V \otimes W, X \otimes W)$, by
$$
\tr_r^W(g)=(\Id_X \otimes \tev_W)(g \otimes \Id_{W^*})(\Id_V \otimes \coev_W)$$
Similarly, a \emph{left partial trace} (with respect to $W$) is the map
$\tr_l^W \colon \Hom_\cat(W \otimes V, W \otimes X) \to \Hom_\cat(V,X)$
defined by 
$$
\tr_l^W(h)=(\ev_W \otimes \Id_X)(  \Id_{W^*}\otimes h)(\tcoev_W \otimes \Id_V).$$

  Let $\lega$ and $\legb$ be objects of $\cat$ and $\ideal$ a right ideal in $\cat$. Given a right ideal $\ideal$,   
a \emph{right (\lega,\legb)-trace} 
%N2: Added
%(or m-trace for short) 
%
is a family of $\kk$-linear functions,
$$\{\mt_V:\Hom_\cat(\lega\otimes V,\legb\otimes V)\rightarrow \FK \}_{V \in \ideal},$$
where $V$ runs over all objects of $\ideal$, and such that the following two conditions hold:
\begin{enumerate}
\item   \textbf{Partial trace property.} If $U\in \ideal$ and $W\in \cat$, then for any $f\in \Hom_\cat(\lega\otimes U\otimes W, \legb \otimes U\otimes W)$ we have
\begin{equation*}%\label{E:VW}
\mt_{U\otimes W}\left(f \right)
=\mt_U \left( \tr_r^W(f)  \right)
%=\mt_U \left( (\Id_{\legb\otimes U}\otimes \tev_W)(f\otimes \Id_{W^*})(\Id_{\lega\otimes U}\otimes \coev_W)   \right).
\end{equation*}
\item   \textbf{(\lega,\legb)-Cyclicity.} If $U,V\in \ideal$, then for any morphisms $f:\lega \otimes V \rightarrow \legb\otimes U $ and $g:U\rightarrow V$  in $\cat$ we have 
\begin{equation*}%\label{E:fggf}
\mt_V((\Id_{\legb}\otimes g) f)=\mt_U(f  (\Id_\lega\otimes g)).
\end{equation*} 
\end{enumerate}
Similarly, if $\ideal$ is a left ideal, then a \emph{left (\lega,\legb)-trace} is a family of linear functions,
$$\{\mt_V:\Hom_\cat(V\otimes \lega,V\otimes \legb)\rightarrow \FK \}_{V\in \ideal},$$
 which satisfies the obvious left partial trace property and the left $(\lega,\legb)$-cyclicity property.   
 %N2:
%N Removed: An \emph{(\lega,\legb)-trace} on an ideal $\ideal$ is a left (\lega,\legb)-trace on $\ideal$ which is also a right (\lega,\legb)-trace.  
 %
%JK3: added below

For brevity, we sometimes simply say m-trace when referring to a left/right $(\lega , \legb )$-trace when the objects $\lega$ and $\legb$ are implicit.  Also, while the m-traces in this paper do not necessarily have the morphism sets of $\ideal$ as their domains, it is a convenient abuse of English to still say the m-trace is defined \emph{on} $\ideal$.

 \begin{remark} 
   %B: When $\lega=\legb=\kk$, 
   When $\lega=\legb=\unit$, a right (resp.\ left) (\lega,\legb)-trace
   is a right (resp.\ left) modified trace as defined in \cite{GPV}.
  \end{remark}

  \subsection{The dual trace}\label{SS:RelatedTrace} As we now
  explain, there is a natural notion of the dual of a right
  m-trace. If $I$ is a full subcategory of $\cat$, define its
  \emph{dual} $I^{*}$ to be the full subcategory with objects
  $I^*=\{V\in\Obj(\cat):V^*\in I\}$.  It is straightforward to check
  if $I$ is a right (resp.\ left) ideal then $I^{*}$ is a left (resp.\
  right) ideal.  If $\mt$ is a right $(\lega,\legb)$-trace $\mt$ on
  $\ideal$, then given $V$ in $\ideal^{*}$ define
  $\mt^{*}_{V}: \Hom_{\cat}(V \otimes \legb^{*} ,
  V \otimes \lega^{*}) \to \kk$ by
  %
  % B: corrected:
  \begin{equation}\label{eq:dual trace}
    \mt^*_{V}(f)=
  \mt_{V^{*}}((\phi_\legb^{-1}\otimes\Id_{V^*})f^*(\phi_\lega\otimes\Id_{V^*}))
  \end{equation}
  % $$\mt^*_{V}(f)=
  % \mt_{V^{*}}((\Id_{V^*}\otimes\phi_\legb^{-1})f^*(\Id_{V^*}\otimes\phi_\lega))$$  
  where $\phi$ is the pivotal structure.  In light of the following
  result we call $\mt^{*}$ the \emph{dual} of $\mt$.
  \begin{lemma}\label{L:dualtrace} Let $\mt$ be a right
    $(\lega,\legb)$-trace $\mt$ on a right ideal, $\ideal$. Then
    $\mt^{*}$ is a left $(\legb^*,\lega^*)$-trace on the left ideal
    $\ideal^*$.
\end{lemma}
\begin{proof}
  To see that the family of linear forms $\mt^*_V$ in \eqref{eq:dual
    trace} is a left m-trace, we need to show that it satifies the left
  partial trace property and the $(\legb^*,\lega^*)$-cyclicity.
  
  Let $U\in\ideal^*$ and $f\in \Hom_{\cat}(W\otimes U \otimes \legb^{*} ,
  W\otimes U \otimes \lega^{*})$ then
  \begin{align*}
    \mt^*_{W\otimes U}\left(f \right)
  &=\mt_{U^*\otimes W^*}\left((\phi_\legb^{-1}\otimes
    \Id_{U^*\otimes W^*})f^*(\phi_\lega\otimes\Id_{U^*\otimes W^*}) \right)\\
    &=\mt_{U^*}\left((\phi_\legb^{-1}\otimes
    \Id_{U^*})\tr_r^{W^*}(f^*)(\phi_\lega\otimes\Id_{U^*}) \right)\\
    &=\mt^*_U \left( \tr_l^W(f)  \right)
  \end{align*}
  where the last equality holds because the right partial trace
  of the dual morphism is the dual of the left partial trace.

  Given $U,V\in \ideal$,
  $f:V \otimes \legb^{*}\rightarrow U \otimes \lega^{*} $ and
  $g:U\rightarrow V$ then
  \begin{align*}
    \mt_V^*((g\otimes\Id_{\lega^{*}}) f)
    &=\mt_{V^*}((\phi_\legb^{-1}\otimes
    \Id_{V^*})f^*(\phi_\lega\otimes\Id_{U^*})(\Id_{\lega}\otimes g^*) )\\
    &=\mt_{U^*}((\Id_{\legb}\otimes g^*)(\phi_\legb^{-1}\otimes
    \Id_{V^*})f^*(\phi_\lega\otimes\Id_{U^*}) )\\
    &=\mt^*_{U}(f(\Id_{\legb}\otimes g) ).
  \end{align*}
\end{proof}
One can analogously define the dual of a left m-trace on a left ideal
$I$ and obtain a right m-trace on $I^{*}$.  
%JK3: Added below
If $\mt$ is a right $(\lega, \legb)$-trace on $\ideal$, then the family of maps $\tilde{\mt}_{V}: \Hom_{\cat}\left(\lega^{**} \otimes V, \legb^{**}\otimes V \right) \to \kk$ given by 
\begin{equation}\label{E:doubledual}
\tilde{\mt}_{V}(f) = \mt_{V}\left( (\phi^{-1}_{\legb} \otimes \Id_{V})f(\phi_{\lega} \otimes \Id_{V}) \right)
\end{equation}
defines a right $(\lega^{**}, \legb^{**})$-trace on $\ideal= \ideal^{**}$.  Similarly for left m-traces.  Furthermore, a straightforward check verifies that the m-trace obtained by dualizing a right or left m-trace twice is related to the original right or left m-trace via \eqref{E:doubledual}.

\subsection{Related traces}
\newcommand{\stu}{{^\#}}
\newcommand{\std}{{_\#}}
We next explain how to construct new m-traces from old.  Assume $\lega_{1}$, $\legb_{1}$, $\lega_{2}$, and $\legb_{2}$ are a fixed list of objects in $\cat$ and we have a fixed morphism $h:\lega_2^*\otimes\legb_2\to \lega_1^*\otimes\legb_1$.  For any object $V$ in $\cat$, the morphism $h$ induces a $\kk$-linear map
\[
h\std: \Hom_\cat( \lega_2\otimes V,\legb_2\otimes V)\to
\Hom_\cat( \lega_1\otimes V, \legb_1\otimes V)
\]
given by 
\[
f \mapsto   (\tev_{\lega_{1}} \otimes \Id_{\legb_{1}} \otimes \Id_{V})     (\Id_{\lega_{1}} \otimes h \otimes \Id_{V})       (\Id_{\lega_{1}} \otimes \Id_{\lega_{2}^{*}} \otimes f)       (\Id_{\lega_{1}} \otimes \tcoev_{\lega_{2}} \otimes \Id_{V}).
\]

\begin{lemma}
  Let $\mt$ be a right $(\lega_1,\legb_1)$-trace on a right ideal $I$.
  Assume we have a fixed morphism $h \in\Hom_\cat(\lega_2^*\otimes\legb_2,\lega_1^*\otimes\legb_1)$.
  Then the family of $\kk$-linear maps $h\stu\mt$, 
\[
\left\{(h\stu\mt)_{V}: \Hom_{\cat}(\lega_{2} \otimes V, \legb_{2}\otimes V) \to \kk  \right\}_{V \in \ideal },
\] defined by 
\[
(h\stu\mt)_{V}(f)=\mt(h\std(f)),
\]  is a right $(\lega_2,\legb_2)$-trace on $I$.  

Furthermore, if
  $h'\in\Hom_\cat(\lega_3^*\otimes\legb_3,\lega_2^*\otimes\legb_2)$
  then ${h'}\stu(h\stu\mt)=(hh')\stu\mt$.
\end{lemma}
\begin{proof} This a straightforward verification using the definition of $h\std$ and $h\stu\mt$.
\end{proof}
\noindent Similarly, a morphism $h:\legb_2\otimes\lega_2^*\to \legb_1\otimes\lega_1^*$ induces a $\kk$-linear map
$$h\std: \Hom_\cat(V\otimes\lega_2,V\otimes \legb_2)\to
\Hom_\cat(V\otimes\lega_1,V\otimes \legb_1)$$
and if $\mt$ is a left $(\lega_1,\legb_1)$-trace on an ideal $I$, then we can analogously define a left $(\lega_2,\legb_2)$-trace $h\stu\mt$.

%N:  Moved to proof:  Using the obvious morphisms as $h$ in the previous lemma along with Lemma~\ref{L:dualtrace} yields the following result.

\begin{prop}\label{P:}  Let $I$ be a right ideal of $\cat$.  Then there are canonical bijections between the following families:
\begin{enumerate}
\item the right $(\lega,\legb)$-traces on $I$,
\item the right $(\legb^*\otimes\lega,\unit)$-traces on $I$,
\item the right $(\unit,\lega^*\otimes\legb)$-traces on $I$,
\item the left $(\legb^*,\lega^*)$-traces on $I^*$,
\item the left $(\lega\otimes\legb^*,\unit)$-traces on $I^*$,
\item the left $(\unit,\legb\otimes\lega^*)$-traces on $I^*$.
\end{enumerate}
\end{prop}
\begin{proof}
  To see the first bijection consider the isomorphism
  $h=\Id_{\alpha^*}\otimes
  \phi_{\beta}^{-1}$.
  %B2: $h=(\Id_{\alpha^*}\otimes
  % \phi_{\beta}^{-1})\gamma_{\beta^*,\alpha}$.
  %
  % Applying this morphism to previous lemma along with
  % Lemma~\ref{L:dualtrace} yields the following result.  The other
  % bijections are similar.
  Applying this morphism to the previous lemma yields the 
  result.  Using Lemma~\ref{L:dualtrace}, the other bijections
  are similar.
\end{proof}
 
 \section{Existence of right and left \texorpdfstring{$(\lega,\legb)$}{(alpha,beta)}-traces}
 
\subsection{Trace tuples}
Let $\cat$ be a pivotal \kt category. In this section we require $\kk$ to be an integral domain.   To simplify exposition we only work with right m-traces and right ideals.  The interested reader can easily formulate the analogous left versions of the definitions and statements.  
  \begin{Df}
Let $\PP$, $\lega$ and $\legb$ be objects of $\cat$.   Let $\eta: \lega\to \PP $ and $\epsilon: \PP \to \legb$ be nonzero morphisms.  
We say $(\PP,\lega, \legb,\eta, \epsilon)$ is a \emph{trace tuple} if the following conditions hold.
\begin{enumerate}
\item The object $\PP$ is absolutely indecomposable and end-nilpotent.  
\item The left $\kk$-modules $\Hom_\cat(\lega,\PP)$ and $\Hom_\cat(\PP,\legb)$ are free and generated by $\eta$ and $\epsilon$, respectively.  
\end{enumerate}
 \end{Df}
 
 Let $(\PP,\lega, \legb,\eta, \epsilon)$ be a trace tuple. 
Consider the following classes of objects:
\begin{align*} 
\ideal_{\lega}&=\{V \in \cat : \text{ there exists } \ss_{V}:  \PP\otimes V \to \lega\otimes V\text{ such that } \ss_{V}(\eta\otimes \Id_{V})=\Id_{\lega\otimes V}\}, \\
\ideal^{\legb}&=\{V \in \cat : \text{ there exists } \tt_{V}:  \legb\otimes V \to \PP\otimes V  \text{ such that } (\epsilon \otimes \Id_{V})\tt_{V}=\Id_{\legb\otimes V}\},\\
\ideal_{\lega}^{\legb}&=\ideal_{\lega}\cap \ideal^{\legb}.
\end{align*}
For each of these, we abuse notation by using the same name for the full subcategory of $\cat$ consisting of objects isomorphic to an object in the given class. 
%N: Moved to proof: The following lemma is a straightforward check using  the definitions.

 \begin{lemma}\label{L:tracetupleideals} If $(\PP,\lega, \legb,\eta, \epsilon)$ is a trace tuple, then $\ideal_{\lega}$, $\ideal^{\legb}$ and $\ideal_{\lega}^{\legb}$ are   right ideals.  
 \end{lemma}
\begin{proof}
To see $\ideal_{\lega}$ is a right idea we need to show it is closed under tensor products and retracts:  by definition if $V\in \ideal_{\lega}$ then there exists $ \ss_{V}:  \PP\otimes V \to \lega\otimes V$ such that $ \ss_{V}(\eta\otimes \Id_{V})=\Id_{\lega\otimes V}$.  Given $W\in \cat$ let $ \ss_{V\otimes W}= \ss_{V}\otimes \Id_W$.  Then  $\ss_{V\otimes W}(\eta\otimes \Id_{V\otimes W})=\Id_{\lega\otimes V\otimes W} $ and this shows  $\ideal_{\lega}$ is closed under tensor products.  To see $\ideal_{\lega}$ is closed under retracts suppose there exists morphisms $f:W\to V$,  $g:V\to W$ such that $g  f=\Id_W$, let $\ss_{W}=(\Id_\alpha\otimes g)\ss_{V}(\Id_\PP\otimes f)$ then $ \ss_{W}(\eta\otimes \Id_{W})=\Id_{\lega\otimes W}$.  
The proof that $\ideal^{\legb}$ and $\ideal_{\lega}^{\legb}$ are   right ideals is similar.  
\end{proof}

We set the following notation.
 When $\PP$ is absolutely indecomposable write $f \mapsto \langle f \rangle$ for the canonical quotient map $\End_{\cat}(\PP) \to \kk$.
 For a trace tuple $(\PP,\lega, \legb,\eta, \epsilon)$ and morphisms $g\in \Hom_\cat(\lega,\PP)$ and $h\in \Hom_\cat(\PP,\legb)$, let $\brk{g}_{\eta},\brk{h}_{\epsilon}\in \kk$ be defined by 
$$g=\brk{g}_{\eta}\eta \text{ and } h=\brk{h}_{\epsilon}\epsilon.$$
 \begin{lemma}\label{L:brk_epsbrk_eta} Let $(\PP,\lega, \legb,\eta, \epsilon)$ be a trace tuple.
  For any  $f\in \End_\cat(\PP) $ the following statements hold: %$\epsilon f = \brk{f} \epsilon$, $f\eta = \brk{f}\eta$ and $\brk{f}=\brk{\epsilon f}_{\epsilon}=\brk{f\eta}_{\eta}$.
 \begin{enumerate}
\item  $\epsilon f = \brk{f} \epsilon$,
\item $f\eta = \brk{f}\eta$,
\item \label{LI3:brk_epsbrk_eta} $\brk{f}=\brk{\epsilon f}_{\epsilon}=\brk{f\eta}_{\eta}$. 
\end{enumerate}
  \end{lemma}
  
 \begin{proof}
Since $\PP $ is absolutely indecomposable, we have $f=\brk{f} \Id_{\PP}+ n$  for $\brk{f} \in \kk$ and $n \in J(\End_{\cat}(P))$.  The first statement then follows once we prove $\epsilon n=0$.  Since $\Hom_\cat(\PP,\legb)$ is a free left $\kk$-module generated by $\epsilon$ we have $\epsilon n=\lambda \epsilon$ for some $\lambda\in \kk$.  But since $P$ is end-nilpotent, $n$ is nilpotent and $n^k =0$ for some $k>0$.  But then $0=\epsilon n^k=\lambda^k \epsilon$ and, hence, $\lambda^k=0$.  Since we are assuming $\kk$ is an integral domain, it follows $\lambda=0$.  This proves the first statement, the second follows analogously. The first two parts of the lemma immediately imply the third statement.
 \end{proof}

\subsection{Existence of m-traces}
\begin{theorem}\label{T:ExistOfabtrace}
 Let   $(\PP,\lega, \legb,\eta, \epsilon)$ be a trace tuple.  Then there exists a right  $(\lega,\legb)$-trace on  $\ideal_{\lega}^{\legb}$ defined  for $V\in \ideal_{\lega}^{\legb}$ and  $f\in \Hom_\cat(\lega\otimes V,\legb\otimes V)$ by
 $$
 \mt_V(f)=\brk{\tr_r^V(\tt_{V}f)}_{\eta}=\brk{\tr_r^V(f \ss_{ V})}_{\epsilon}
 $$
where $\ss_{ V}: \PP\otimes V \to \lega\otimes V$ and $\tt_{V}:\legb\otimes V \to \PP\otimes V$ are any morphisms satisfying $\ss_{V}(\eta\otimes \Id_{V})=\Id_{\lega\otimes V}$ and $(\epsilon \otimes \Id_{V})\tt_{V}=\Id_{\legb\otimes V} $.
\end{theorem}
\begin{proof}
First, we note $\mt_{V}$ is $\kk$-linear and the morphisms $\ss_{V}$ and $\tt_{V}$ exist because $V$ is assumed to lie in $\ideal_{\lega}^{\legb}$.  Next, let $\ss_{ V}$ and $\tt_{V}$ be any such morphisms.  Then,
\begin{align*}
\brk{\tr_r^V\big(\tt_{V}f\big)}_{\eta}=\brk{\tr_r^V\big(\tt_{V}f[\Id_{\lega\otimes V}]\big)}_{\eta}&=\brk{\tr_r^V\big(\tt_{V}f [\ss_{V}(\eta\otimes \Id_{V})]\big)}_{\eta}\\
&=\brk{\tr_r^V\big(\tt_{V}f \ss_{V}\big)\eta}_{\eta}\\
&=\brk{\epsilon\tr_r^V\big(\tt_{V}f \ss_{V}\big)}_{\epsilon}\\
&=\brk{\tr_r^V\big([\epsilon\otimes \Id_{V}] \tt_{V}f \ss_{V}\big)}_{\epsilon}\\
&=\brk{\tr_r^V\big(f \ss_{V}\big)}_{\epsilon},
\end{align*}
where  the fourth equality comes from Lemma \ref{L:brk_epsbrk_eta} part \eqref{LI3:brk_epsbrk_eta}.  Thus, $\mt_V(f)$ is independent of the choice of $\ss_{V}$ or $\tt_{V}$. 

Next we show this family of functions satisfies the partial trace property.  Let $U\in \ideal_{\lega}^{\legb}$, $W\in \ob$ and $f\in \Hom_\cat(\lega\otimes U\otimes W, \legb \otimes U\otimes W)$.  Since $U\in \ideal^{\legb}$ there exists $\tt_{U}:\legb\otimes U \to \PP\otimes U $ such that $ (\epsilon \otimes \Id_{U})\tt_{U}=\Id_{\legb\otimes U}$.  Choose $\tt_{U\otimes W}$ to be equal to $\tt_{U}\otimes \Id_W$ then $ (\epsilon \otimes \Id_{U\otimes W})\tt_{U\otimes W}=\Id_{\legb\otimes U\otimes W}$.  Therefore, we can use $\tt_{U\otimes W}$ to define $\mt_{U\otimes W}$ and we see
\begin{align*}
\mt_{U\otimes W}(f)&=\brk{\tr_r^{U\otimes W}\big(\tt_{U\otimes W}f\big)}_{\eta}\\
&=\brk{\tr_r^{U\otimes W}\big((\tt_{U}\otimes \Id_W) f\big)}_{\eta}\\
&=\brk{\tr_r^{U}\big(\tt_{U}(\tr_r^{W} f)\big)}_{\eta}\\
&=\mt_{U}\big(\tr_r^{W} (f)\big).
\end{align*}

To prove the $(\lega,\legb)$-cyclicity property, let $f:\lega \otimes V \rightarrow \legb\otimes U $ and $g:U\rightarrow V$.  Then, 
\begin{align*}
\mt_V((\Id_{\legb}\otimes g) f)%=\brk{\tr_r^V\big(\tt_{V}(\Id_{\legb}\otimes g) f\big)}_{\eta}
=\brk{\tr_r^V\big((\Id_{\legb}\otimes g) f\ss_{V}\big)}_{\epsilon}
&=\brk{\tr_r^U\big( f\ss_{V}(\Id_{\PP}\otimes g)\big)}_{\epsilon}\\
&=\brk{\tr_r^U\big((\epsilon \otimes \Id_{U}) \tt_{U} f\ss_{V}(\Id_{\PP}\otimes g)\big)}_{\epsilon}\\
&=\brk{\tr_r^U\big( \tt_{U} f(\Id_{\lega}\otimes g)\big)}_{\eta},
\end{align*}
where the first equality comes from the definition of the trace, the
second from the properties of the pivotal structure, the third from
the definition of $\ideal^{\legb}$, and the last from
Lemma \ref{L:brk_epsbrk_eta} part \eqref{LI3:brk_epsbrk_eta}.
\end{proof}

\subsection{Uniqueness}  It is of particular interest when one or both of $\lega$ and $\legb$ are the unit object.  In this case we have the following uniqueness result.
\begin{theorem}\label{T:UniqunessofTrace}
Suppose $(\PP,\lega, \legb,\eta, \epsilon)$ is a trace tuple with $\PP\in \ideal_{\lega}^{\legb}$.  

\begin{enumerate}
\item If $\legb =\unit$, then the right $(\lega , \unit )$-trace on $\ideal_{\lega}^{\unit}$ is unique up to a scalar.  Specifically, if $ \mt'$ is a right $(\lega , \unit )$-trace  on  $\ideal_{\lega}^{\unit}$ and $\mt$ is the right  $(\lega , \unit )$-trace defined by Theorem~\ref{T:ExistOfabtrace}, then 
$$
 \mt'= \mt'_\PP(\eta\otimes \epsilon)\mt.
$$
 Moreover, $\mt_\PP(\eta\otimes \epsilon)=1$.  
\item If $\lega =\unit$, then the right $(\unit , \legb )$-trace on $\ideal_{\unit}^{\legb}$ is unique up to a scalar.  Specifically, if $ \mt'$ is a right $(\unit , \legb )$-trace  on  $\ideal_{\unit}^{\legb}$ and $\mt$ is the right  $(\unit , \legb )$-trace defined by Theorem~\ref{T:ExistOfabtrace}, then 
$$
 \mt'= \mt'_\PP(\epsilon\otimes \eta)\mt.
$$
 Moreover, $\mt_\PP(\epsilon\otimes \eta)=1$.  
\end{enumerate}

\end{theorem}
\begin{proof}
Let  $ \mt'$ be a right $(\lega , \unit )$-trace on  $\ideal_{\lega}^{\unit}$.  If $f\in \Hom_\cat(\lega\otimes V,V)$ then
\begin{align*}
  \mt'_V\big(f\big)=\mt'_V\big((\epsilon \otimes \Id_{V})\tt_{V}f)
  &=\mt'_{\PP\otimes V}((\tt_{V}f)(\Id_\lega \otimes \epsilon \otimes \Id_{V})\big)\\
%B: MaCo10:
  &=\mt'_{\PP}\big(\tr_r^V\left(\left( \tt_{V}f)(\Id_\lega \otimes \epsilon \otimes \Id_{V}\right) \right)\big)\\
  &=\mt'_{\PP}\big(\tr_r^V\left( \tt_{V}f\right)\left(\Id_\lega \otimes \epsilon \right) \big)\\
  % &=\mt'_{\PP}\big(\tr_r^V\left(\left( \tt_{V}f)(\Id_\lega \otimes \epsilon \right) \right)\big)\\
&=\mt'_{\PP}\big(\brk{\tr_r^V(\tt_{V}f)}_{\eta}\eta(\Id_\lega \otimes \epsilon)\big)\\
%&=\brk{\tr_r^V(\tt_{V}f)}_{\eta}\mt'_{\PP}\big(\eta\otimes \epsilon\big)\\
&=\mt_V(f)\mt'_{\PP}\big(\eta\otimes \epsilon\big).
\end{align*}
The first equality comes from the fact $(\epsilon \otimes \Id_{V})\tt_{V}=\Id_{V}$ by definition of $\ideal^\unit$; the second from strictness and $(\lega,\unit)$-cyclicity, the third from the partial trace property and the last two from the definitions of the $\eta$-bracket and the trace $\mt$, respectively.  

Finally, using the same properties along with Lemma \ref{L:brk_epsbrk_eta} yields
\begin{align*}
\mt_\PP(\eta\otimes \epsilon)
=\brk{\tr_r^\PP\big(\tt_{\PP}(\eta\otimes \epsilon)\big)}_\eta
=\brk{(\Id_\PP\otimes \epsilon)\tt_{\PP}\eta}_\eta
=\brk{(\epsilon\otimes \epsilon)\tt_{\PP}}_\epsilon=\brk{\epsilon}_\epsilon=1.
\end{align*}
%B: MaCo10:
where the second equality is because $\tr_r^\PP\big(\tt_{\PP}(\eta\otimes \epsilon)\big)=\tr_r^\unit\big((\Id_\PP\otimes \epsilon)\tt_{\PP}(\eta\otimes\Id_\unit)\big)=(\Id_\PP\otimes \epsilon)\tt_{\PP}\eta$.
%B we could put smthing like tr_r(f(id@g))=tr_r((id@g)f) 
The proof of the second statement is entirely analogous.
\end{proof}

For short, when $\legb=\unit$ we say a right (resp.\ left) $(\lega,\unit)$-trace on a right (resp.\ left) ideal $\ideal$ is a \emph{right (resp.\ left) $\lega$-trace} on $\ideal$.

\subsection{A handy lemma}
The following lemma will be useful in what follows.

\begin{lemma}\label{L:cutP}
  Let $\mt$ be the right trace associated to a trace tuple
  $(\PP,\lega, \legb,\eta, \epsilon)$ as in Theorem~\ref{T:ExistOfabtrace} and let $V$ be an object in $\ideal_{\lega}^{\legb}$.   Then:
\begin{enumerate}
\item For any
  $f\in\Hom_\cat(\PP\otimes V,\legb\otimes V)$, one has
  $\mt_V(f(\eta\otimes\Id_V))=\brk{\tr^V_r(f)}_\epsilon$.
\item  For any $g\in\Hom_\cat(\lega\otimes V,\PP\otimes V)$, one has
  $\mt_V((\epsilon\otimes\Id_V)g)=\brk{\tr^V_r(g)}_\eta$.
\item  For any $h\in\Hom_\cat(\PP\otimes V,\PP\otimes V)$,
  $\mt_V((\epsilon\otimes\Id_V)h(\eta\otimes\Id_V))=\brk{\tr^V_r(h)}$.
\end{enumerate}

\end{lemma}
\begin{proof}
  From definitions and Lemma~\ref{L:brk_epsbrk_eta}, we have
\[
\mt_V(f(\eta\otimes\Id_V))=\brk{\tr_r^V(\tt_Vf(\eta\otimes\Id_V))}_\eta 
  =\brk{\tr_r^V(\tt_Vf)\eta}_\eta=\brk{\epsilon\tr_r^V(\tt_Vf)}_\epsilon=\brk{\tr^V_r(f)}_\epsilon.
\]  Similarly, 
\[
\mt_V((\epsilon\otimes\Id_V)g)=\brk{\tr_r^V((\epsilon\otimes\Id_V)g\ss_V)}_\epsilon=\brk{\epsilon\tr_r^V(g\ss_V)}_\epsilon=\brk{\tr_r^V(g\ss_V)\eta}_\eta=\brk{\tr^V_r(g)}_\eta.
\]  The third statement is proven in a similar fashion.
\end{proof}

\subsection{Left and right compatibility}  If the reader formulates the theory for left ideals and m-traces, then the result is compatible with the right version, as we next explain.
\begin{Df}
  Let $\mt^l$ be a left $(\lega, \legb)$-trace on $I^l$ and $\mt^r$ be
  a right $(\lega, \legb)$-trace on $I^r$.  We say $\mt^l$ and
  $\mt^r$ are \emph{compatible} if for any $(V,W)\in I^l\times I^r$, and for any 
  $f\in \Hom_\cat(V\otimes\lega\otimes W,V\otimes\legb\otimes W)$,
  $$\mt^r_W(\tr_l^V(f))=\mt^l_V(\tr_r^W(f)).$$
\end{Df}
\begin{prop}
  The left and right trace associated to a trace tuple
  $(\PP,\lega, \legb,\eta, \epsilon)$ are compatible.
\end{prop}
\begin{proof}
 Since $V$ is in the left ideal $I_\lega$ and $W$ is in the right ideal $I^\legb$, there exists $\ss_V:V\otimes \PP\to V\otimes \lega$ and $\tt_W:\legb\otimes W\to\PP\otimes W$ such that
  $$f=(\Id_V\otimes\epsilon\otimes\Id_W)(\Id_V\otimes\tt_W)f(\ss_V\otimes\Id_W)(\Id_V\otimes\eta\otimes\Id_W).$$ 
%N3:  but then we have
Then by Lemma \ref{L:cutP} we have
$$\mt^l_V(\tr_r^W(f))=\brk{\tr_l^V(\tr_r^W((\Id_V\otimes\tt_W)f(\ss_V\otimes\Id_W)))}=\mt^r_W(\tr_l^V(f)).$$% by Lemma \ref{L:cutP}.
\end{proof}

\section{Examples}\label{S:Examples}

In addition to the known examples of unimodular m-traces in the literature (e.g.\ \cite{GKP1,GKP2,BBGa,GR}), we have the following non-unimodular m-traces.

\subsection{The toy example}\label{SS:ToyExample}  Let $S$ be an absolutely irreducible object in $\cat$ and set $\PP=\lega=\legb=S$, and let $\epsilon: \PP \to \legb$ and $\eta: \lega \to \PP$ be the identity maps.  Then  $(\PP,\lega, \legb,\eta, \epsilon)$ is a trace tuple and  $\ideal_{\lega}^{\legb}=\cat$.  
Also, the $(\lega,\legb)$-trace of Theorem \ref{T:ExistOfabtrace} is  given by 
$\mt_{V}(f) = \langle \tr_{r}^{V}(f) \rangle$ for all $f \in \Hom_{\cat}(S \otimes V, S \otimes V)$.

\subsection{Quantized enveloping algebras} In this subsection let $\kk = \C (q)$, where $q$ is an indeterminate. We follow standard conventions without elaboration. The reader may consult \cite{RAGS,J,CP} for further details.  Let $\fg$ be a complex semisimple Lie algebra and let $\fh \subseteq \fb \subseteq \fg$ be a fixed choice of Cartan and Borel subalgebras, respectively.  Let $U_{q}(\fh ) \subseteq U_{q}(\fb ) \subseteq U_{q}(\fg )$ be the corresponding quantized enveloping algebras over $\kk$. 

We order elements of the weight lattice, $\Lambda$, using the usual dominance order determined by our choice of $\fb$.  If $L$ is a finite-dimensional simple 
$U_{q}(\fg)$-module 
then there is a unique maximal nonzero weight space.  Let $\lambda \in \Lambda$ be the highest weight and let $L_{\lambda}$ denote the corresponding $1$-dimensional $\lambda$-weight space.  Similarly, $L$ has a unique lowest nonzero weight space, $L_{\alpha}$.  By restriction we can view $L$ as a $U_{q}(\fb )$-module, and then $L_{\lambda}$ is the simple socle, $L_{\alpha}$ is the simple head, and $L$ is cyclically generated by $L_{\alpha}$.  In particular, $L$ is an absolutely indecomposable $U_{q}(\fb )$-module.

In short, the previous paragraph shows the canonical projection and inclusion maps $\varepsilon: L \to L_{\alpha}$ and $\eta: L_{\lambda} \to L$ make $(L, L_{\lambda}, L_{\alpha}, \epsilon, \eta)$ into a trace tuple in the category of finite-dimensional $U_{q}(\fb )$-modules (which is known to be a pivotal tensor  $\kk$-category).  A similar example holds in the non-quantum case as well. 

Now let $U_{\zeta}(\fg )$ be the restricted specialization of the quantized enveloping algebra at $\zeta \in \C$, a primitive, odd $\ell$th root of unity.  We assume $\ell$  is greater than the Coxeter number for $\fg$ and is not divisible by $3$ if $\fg$ has a direct summand of type $G_{2}$.  
For a dominant integral $\lambda \in\Lambda $, let $H^{0}_{\zeta}(\lambda)$, $V_{\zeta}(\lambda)$, and $T_{\zeta}(\lambda)$ denote the induced, Weyl, and tilting $U_{\zeta}(\fg )$-modules of highest weight $\lambda$ with respect to the fixed choice of a Borel subalgebra.  Then $H^{0}_{\zeta}(\lambda)$ and $V_{\zeta}(\lambda)$ are absolutely irreducible and $T_{\zeta}(\lambda)$ is absolutely indecomposable.  Furthermore, since $\Hom_{U_{\zeta}(\fg )}(V_{\zeta}(\lambda), T(\lambda)) = \C$ and $\Hom_{U_{\zeta}(\fg )}(T_{\zeta}(\lambda), H_{\zeta}^{0}(\lambda))=\C$, there are maps $\eta: V_{\zeta}(\lambda) \to T(\lambda)$ and $\epsilon: T_{\zeta}(\lambda) \to H^{0}_{\zeta}(\lambda)$ which make $\left(T_{\zeta}(\lambda),  H^{0}_{\zeta}(\lambda),  V_{\zeta}(\lambda), \epsilon, \eta\right)$ into a trace tuple.  A parallel example  exists for semisimple algebraic groups over an algebraically closed field.

%JK3: Added below
It would be interesting to determine the ideals defined by Lemma~\ref{L:tracetupleideals} using these trace tuples.

\subsection{Projective objects} 
%Added abelian below --JK
 In this section $\kk$ is assumed to be an algebraically closed field and $\cat$ is a locally-finite, pivotal, abelian, tensor $\kk$-category.  
%Hide the next sentence --JK
%In particular, in $\cat$ every simple object is absolutely simple by Schur's Lemma and every indecomposable object is absolutely indecomposable and end-nilpotent by Fitting's Lemma.  
As remarked in Subsection \ref{SS:IntroStatMainResults}, such categories include a wide range of examples. %N This subsection implies these all admit unique nontrivial right m-traces. 

\begin{lemma}\label{P:Projectivetracetuple} If $\PP$ is the projective cover of a simple object $\legb$ in $\cat$, then there is a unique simple object $\lega$, epimorphism $\epsilon: \PP \to \legb$, and monomorphism $\eta: \lega \to \PP$ such that $(\PP,\lega, \legb,\eta, \epsilon)$ is a trace tuple. 
\end{lemma}

\begin{proof} If $\PP $ is the projective cover of $\legb$, then it is absolutely indecomposable, end-nilpotent, and has simple socle $\lega$ by Section~\ref{SS:ProjInj}.  Set $\epsilon: \PP  \to \legb$ and $\eta: \lega \to \PP $ to be the canonical projection and inclusion, respectively.  Then $(\PP,\lega, \legb,\eta, \epsilon)$ is a trace tuple as described in the claim.
\end{proof}

Recall $\Proj(\cat)$ denotes the ideal of projective objects in $\cat$. 

%Fixed Statement and proof below.  Note, the proof uses epi and mono, obviously.  However, you can make a trace tuple with three ``adjacent'' Kac modules of gl(1|1) using non-identity maps where the ideal doesn't contain the projectives (or any nonzero modules, in fact). --JK

\begin{lemma}\label{L:Prinjectives} Let  $(\PP,\lega, \legb,\eta, \epsilon)$ be a trace tuple where $\varepsilon$ is an epimorphism and $\eta$ is a monomorphism. Then, $\Proj(\cat ) \subseteq \ideal_{\lega}^{\legb}$. In particular, if $\cat$ contains a projective object, then $\ideal_{\lega}^{\legb}$ is nonempty.
\end{lemma}

\begin{proof}  Let $Q$ be a projective object in $\cat$.   
Since the tensor functor is exact, the morphism $\epsilon \otimes \Id_Q: \PP \otimes Q \to \legb \otimes Q$ 
is an epimorphism.  Since $Q$ is projective and $\Proj (\cat)$ is an ideal, it follows $\legb \otimes Q$ is projective, 
and so the morphism $\epsilon \otimes \Id_Q$ 
splits.  Therefore, $Q$ is an object of $\ideal^{\legb}$. 
 Similarly, $\eta \otimes \Id_Q : \lega \otimes Q \to \PP \otimes Q$ is a monomorphism and $\lega \otimes Q$ is projective (hence injective), so the morphism $\eta \otimes \Id_Q$ again splits  and  $Q$ is an object of $\ideal_{\lega}$.    Taken together this shows $Q \in \ideal_{\lega}^{\legb}$ 
\end{proof}

Note that examples can be found which show the previous result can fail if there are no assumptions on the trace tuple.

%Revised the statement and proof below. --JK

\begin{lemma}\label{P:ProjectiveandInvertible}  Given any trace tuple  $(\PP,\lega, \legb,\eta, \epsilon)$ where $\PP$ is projective and either $\lega$ or $\legb$ is invertible.  Then, $\ideal_{\lega}^{\legb} \subseteq \Proj(\cat)$.   
\end{lemma}

\begin{proof}  We do only the case when $\legb$ is invertible as the other case is similar.  Let $V \in \ideal_{\lega}^{\legb}$.  Then by definition $\epsilon \otimes \Id_V: \PP \otimes V \to \legb \otimes V$ splits. But $\PP$ is in the ideal $\Proj (\cat)$, so $\PP \otimes V$ is in $\Proj (\cat )$ and, hence, $\legb \otimes V$ is in $\Proj (\cat )$.  Consequently, $\legb^{*}\otimes\legb \otimes V \cong V$ is an object of $\Proj (\cat)$.   
\end{proof}

Note, if $S$ is an  absolutely irreducible, projective object, then the toy example 
of Subsection \ref{SS:ToyExample}
 shows the previous result could fail if there are no assumptions on $\lega$ and $\legb$.

The following theorem summarizes the outcome of the previous lemmas.

\begin{theorem}\label{T:ExistTraceTupleProj} If $\PP$ is the projective cover of a simple object in $\cat$, then there are unique simple objects $\lega$ and $\legb$ and morphisms $\epsilon: \PP \to \legb$ and $\eta: \lega \to \PP$ such that $(\PP,\lega, \legb,\eta, \epsilon)$ is a trace tuple and $\Proj (\cat) \subseteq \ideal_{\lega}^{\legb}$.  Moreover, if either $\lega$ or $\legb$ is invertible, then $\Proj (\cat )=\ideal _{\lega}^{\legb}$.
\end{theorem}  

The next result demonstrates the  %N3: $(\lega , \legb )$-trace 
m-trace 
defined by the previous result combined with Theorem~\ref{T:ExistOfabtrace} is nontrivial.  
Specifically, given $Q\in \Proj (\cat ) $, one can choose $V$ (since it is arbitrary) so   
  $\Hom_{\cat}(\lega \otimes Q, V)$ is nontrivial.  
Consequently, for any $Q \in \Proj (\cat)$ the next theorem shows both  
$\Hom_{\cat}(\lega \otimes Q, \legb \otimes Q)$ and $\mt_{Q}$ are nonzero.

\begin{theorem}\label{T:nondegeneratepairing} 
Let $(\PP,\lega, \legb,\eta, \epsilon)$ be the trace tuple 
 given by the previous theorem. 
 Let $\mt$ be the right trace given by Theorem~\ref{T:ExistOfabtrace}.  
Then for any $Q\in  \Proj (\cat) \subseteq \ideal_{\lega}^{\legb}$ and $V\in \cat$ the map
$$
\Hom_\cat(V,\legb\otimes Q) \times \Hom_\cat(\lega\otimes Q,V)\to \FK \text{ given by } (g,f)  \mapsto \mt_Q(gf) 
$$
is a non-degenerate pairing.  
\end{theorem}
\begin{proof}
  From Lemma~\ref{L:Prinjectives}
  $\Proj (\cat) \subseteq \ideal_{\lega}^{\legb}$ so the function
  exists.  We next show its right kernel is trivial (the proof for the
  left kernel is similar).  If $f\in \Hom_\cat(\lega\otimes Q,V)$ is
  not zero, then
  $f'=(f\otimes\Id_{Q^{*}})
  (\Id_{\lega}\otimes\coev_Q)\in\Hom_\cat(\lega,V\otimes Q^*)$ is a
  non zero map from $\lega$ to the projective object
  $V \otimes Q^{*}$.  Since projective covers (hence injective envelopes) are unique, $\PP$ is the unique indecomposable
  projective object 
with $\lega$ as a subobject, the map $f'$ factors
  through an indecomposable summand of $V\otimes Q^*$ which is
  isomorphic to $\PP$.  That is, there are morphisms
  $\iota: \PP \to V \otimes Q^{*}$ and $p: V \otimes Q^{*} \to \PP$ such
  that $p \iota = \Id_{\PP}$ and $f' = \iota\eta$.

Let $g \in \Hom_\cat(V,\legb\otimes Q)$ be given by $g = (\epsilon \otimes \Id_{Q})(p \otimes \Id_{Q})(\Id_{V} \otimes \tcoev_{Q})$.  Then $gf = (\epsilon \otimes \Id_{Q})f''$ where $f'' \in \Hom_{\cat}(\lega \otimes Q, \PP \otimes Q)$ is given by $f''= (p \otimes \Id_{Q})(\Id_{V} \otimes \tcoev_{Q})f$. The first of the following equalities holds by Lemma~\ref{L:cutP}:
\[
\mt_{Q}(gf) = \langle \tr_{r}^{Q}(f'') \rangle_{\eta} = \langle pf' \rangle_{\eta} = \langle p\iota \eta \rangle_{\eta} = \langle \eta \rangle_{\eta} = 1.
\]
\end{proof}

It would be interesting to understand when the pairing in the above theorem is non-degenerate on the whole ideal $\ideal_{\lega}^{\legb}$.  
Combining Theorems~\ref{T:ExistTraceTupleProj} and \ref{T:UniqunessofTrace} with the previous result immediately yields the following corollary.  
\noindent  %N: There is also an analogous unique, nontrivial left  m-trace  on $\Proj (\cat )$.

%Changed nontrivial to nondegenerate.

\begin{corollary}
Let $\kk$ be an algebraically closed field and $\cat$ be a locally-finite, pivotal, $\kk$-tensor category which has enough projectives.
 Let $\PP$ be the projective cover of $\unit$ and let $\lega$ be the socle of $\PP$.  This data determines a unique (up to scalar), nondegenerate (in the sense of Theorem~\ref{T:nondegeneratepairing}) right $\lega$-trace on $\Proj(\cat )$.
\end{corollary}

\subsection{Ambidextrous objects} In earlier work the authors introduced the notion of a right ambidextrous object and the associated right m-trace.  We now explain how that construction is a special case of the one introduced here. Let $\cat$ be a ribbon category, $S$ be an absolutely irreducible object, and let $\epsilon=\tev : S \otimes S^{*} \to \unit$ and $\eta = \coev : \unit \to S \otimes S^{*}$.  Let $S \otimes S^{*} = \oplus_{i} W_{i}$ be the decomposition of $S$ into indecomposable objects.  Then $S$ is right ambi in the sense of \cite{GKP2} if and only if there is an $i$ such that the restriction of $\epsilon$ and $\eta$ to $\PP:=W_{i}$ makes $(\PP, \unit , \unit , \epsilon, \eta)$ into a trace tuple (this follows directly from Lemma 3.1.1 and Theorem 3.1.3 of \cite{GKP2}).  In which case 
%N3:  $\ideal_{\unit}^{\unit}=\ideal_{S}$ and the trace 
$\ideal_{\unit}^{\unit}$ equals the ideal generated by $S$ and the m-trace 
defined here agrees with the one defined therein.

\section{Twisted Calabi-Yau Categories}\label{S:twistedCYcategories}    In this section we continue to assume $\kk$ is a field and $\cat$ is a $\kk$-linear category.

\subsection{Twisted Calabi-Yau Categories}\label{SS:twistedCYcategories} Next we introduce the notion of a twisted Calabi-Yau category.

\begin{Df}\label{D:twistedCYcategory} Let $F,G: \cat \to \cat$ be fixed endofunctors of a category, $\cat$.  Then $\cat$ is an \emph{$(F,G)$-twisted Calabi-Yau category} if it is equipped with a family of $\kk$-linear maps 
\[
\left\{\mt_{U} : \Hom_{\cat}(F(U),G(U)) \to \kk \right\}_{U\in \cat}
\] such that the following properties hold:
\begin{enumerate}
\item   \textbf{Non-degeneracy.} For any objects $U,V$ in $\cat$, the pairing
\begin{equation*}\Hom_\cat(V,G(U)) \times \Hom_\cat(F(U),V)\to \FK \text{ given by } (g,f)  \mapsto \mt_{U}(gf) 
\end{equation*} is non-degenerate.
\item   \textbf{Cyclicity.} For any objects $U,V$ in $\cat$ and any morphisms $f: F(V) \rightarrow  G(U) $ and $g:U\rightarrow V$ in $\cat$, we have 
\begin{equation*}%\label{E:fggf}
\mt_V(G(g) f)=\mt_U(f F(g)).
\end{equation*} 
\end{enumerate}

\end{Df}

When $\cat$ is locally-finite the non-degeneracy condition provides a canonical vector space isomorphism, 
\[
 \Hom_{\cat}(F(U),V) \cong \Hom_{\cat}(V,G(U))^{*},
\] which is functorial in both $U$ and $V$.
 
This notion generalizes existing constructions. For example, a $(\Id_{\cat}, \Id_{\cat})$-twisted Calabi-Yau category is nothing but a Calabi-Yau category.  If a category $\cat$ is a $(\Id_{\cat}, G)$-twisted Calabi-Yau category, then $G$ is a right Serre functor in the sense of Bondal-Kapranov \cite{BK}.

 In the special case when $\cat$ is a category with a single object, $*$, then being an $(F,G)$-twisted Calabi-Yau category is equivalent to having a $\kk$-linear map $\mt :  \End_{\cat}(*) \to \kk$ which satisfies $\mt (g(a)b)=\mt (af(b))$ for fixed algebra endomorphisms $f,g: \End_{\cat}(*) \to \End_{\cat}(*)$ along with the requirement the pairing $(a,b) \mapsto \mt (ab)$ be nondegenerate.  In this way it generalizes the well known fact a Calabi-Yau structure on a category with a single object is equivalent to the notion of a symmetric Frobenius algebra.

\subsection{A twisted Calabi-Yau structure on $\Proj (\cat )$}
In this section we assume $\kk$ is an algebraically closed field and $\cat$ is a locally-finite, pivotal, $\kk$-tensor category with enough projectives.

If $X$ is an fixed object of $\cat$, then we write $F_{X}$ for the endofunctor $X \otimes -$. If $\PP$ is an indecomposable projective object in $\cat$, then there is the corresponding trace tuple $(\PP,\lega, \legb,\eta, \epsilon)$ and
right 
 $(\lega, \legb)$-trace, $\mt$, on $\Proj (\cat )$ given by Theorem~\ref{T:ExistTraceTupleProj}.  Combining this with Theorem~\ref{T:nondegeneratepairing} yields the following result.  

\begin{theorem}\label{T:twistedKYonProjforlocallyfinite}  Let $\PP$ be an indecomposable projective object in $\cat$.  The corresponding trace tuple $(\PP,\lega, \legb,\eta, \epsilon)$ and  
right  
$(\lega, \legb)$-trace, $\mt$, makes $\Proj (\cat)$ into an $(F_{\lega}, F_{\legb})$-twisted Calabi-Yau category.
\end{theorem}

As an application, if we take $P$ to be the injective hull of $\unit$ and $\legb$ is the simple head of $P$, then $\Proj (\cat)$ is an $(\Id_{\cat}, F_{\legb})$-twisted Calabi-Yau category and, hence, $F_{\legb}$ is a right Serre functor on $\Proj (\cat)$.  
We also have the following special case of the previous theorem. Recently Gainutdinov-Runkel \cite{GR} obtained the same result under the assumption $\cat$ is finite and factorisable.  

\begin{corollary}\label{T:twistedCYonProjforfinite} Assume $\cat$ is a locally-finite, pivotal, unimodular, $\kk$-tensor category with enough projectives. Then $\Proj(\cat )$ is a Calabi-Yau category.
\end{corollary}

%Added explanation about tensoring by $\lega$ --JK

We end by noting the following generalization of the single object example from the previous section.  If $\cat$ is a finite tensor category as in \cite{EGNO}, then it has finitely many indecomposable projectives $P_{1}, \dotsc , P_{t}$ and $Q = \oplus_{i=1}^{t} P_{i}$ is a projective generator. Let $P_{0}$ be the projective cover of $\unit$ and let $\lega$ be the socle of $P_{0}$.  Then $\lega$ is invertible and, up to isomorphism, tensoring by $\lega$ permutes the indecomposable projectives.  Consequently, $F_{\lega}(Q) = \lega \otimes Q \cong Q$ and the 
right $\lega$-trace 
 on $\Proj  (\cat )$ defines a $\kk$-linear map $\mt : \End_{\cat}(Q) \to \kk$ which makes $\End_{\cat}(Q)$  a Frobenius algebra.  Moreover, the algebra endomorphism of $\End_{\cat}(Q)$ induced by the functor $F_{\lega}$ is the Nakayama automorphism.  Thus $\cat$ is equivalent to the category of finite-dimensional modules over a finite-dimensional Frobenius algebra.

%\acknow{Research of the first author has been partially supported by the NSF grants DMS-1452093 and DMS-1664387.  Research of the second author was partially supported by a Simons Foundation Collaboration Grant.}

\makeatletter
\makeatother

\bibliographystyle{halpha}

\end{document}